\documentclass[10pt]{article}
\usepackage[margin=1.5in]{geometry}  
\usepackage{amsmath}
\usepackage{amsfonts}
\usepackage{amssymb}
\usepackage{amsthm}
\usepackage{cases}
\usepackage{stmaryrd}
\usepackage{latexsym}
\usepackage{mathrsfs}
\usepackage{graphicx}
\usepackage{verbatim}

\newtheorem{thm}{Theorem}[section]
\newtheorem{cor}[thm]{Corollary}
\newtheorem{lem}[thm]{Lemma}
\newtheorem{prop}[thm]{Proposition}

\newtheorem{rem}[thm]{Remark}

\numberwithin{equation}{section}

\def\C{\mathbb{C}}
\def\E{\mathbb{E}}
\def\N{\mathbb{N}}
\def\P{\mathbb{P}}

\def\R{\mathbb{R}}

\def\Z{\mathbb{Z}}

\def\BB{\mathcal{B}}

\def\SS{\mathcal{S}}

\def\LLL{\mathscr{L}}

\def\supp{\text{\rm supp}}

\def\range{\rm ran}
\def\sgn{\rm sgn}

\def\lan{\langle}
\def\ran{\rangle}

\def\ra{\rightarrow}
\def\bs{\backslash}
\def\ol{\overline}

\def\al{\alpha}

\def\ep{\epsilon}

\def\la{\lambda}

\def\si{\sigma}

\def\om{\omega}
\def\Om{\Omega}
\def\de{\delta}
\def\De{\Delta}

\def\vp{\varphi}

\begin{document}

\nocite{*}

\title{Scattering Theory of Schr\"{o}dinger Operators with Random Sparse Potentials}

\author{Zhongwei Shen\footnote{Email: zzs0004@auburn.edu}\\Department of Mathematics and Statistics\\Auburn University\\Auburn, AL 36849\\USA}

\date{}

\maketitle

\begin{abstract}
In this paper, we study the scattering theory of a class of continuum Schr\"{o}dinger operators with random sparse potentials. The existence and completeness of wave operators are proven by establishing the uniform boundedness of modified free resolvents and  modified perturbed resolvents, and by invoking a previous result on the absence of absolutely continuous spectrum below zero. \\
\textbf{Keywords.} Schr\"{o}dinger operator, random sparse potential, wave operator, completeness.\\
\textbf{2010 Mathematics Subject Classification.} Primary 47A40; Secondary 47A10, 81Q10.
\end{abstract}


\section{Introduction}\label{sec-intro}

Absolutely continuous (a.c.) spectrum is one of the main topics of the theory of Schr\"{o}dinger operators with random potentials. While it is very hard to prove the existence of a.c. spectrum for Schr\"{o}dinger operators on $L^{2}(\R^{d})$ or $\ell^{2}(\Z^{d})$ with stationary random potentials (see e.g. \cite{Si00}), many Schr\"{o}dinger operators with non-stationary random potentials have been proven to exhibit a.c. spectrum (see e.g. \cite{DSS85,FHS10,KKO00,KS01}). Moreover, since the work of Krishna \cite{Kri90}, wave operators have been established for decaying random potentials (see e.g. \cite{Bou02,Bou03,Kri90,Kri12,RS03}) and random sparse potentials (see e.g. \cite{FS05,HK00,JP09,Ki02,Kru04,Po08}).

In the present paper, we study the existence and completeness of wave operators of a class of continuum Schr\"{o}dinger operators with random sparse potentials, that is,
\begin{equation}\label{main-model}
H_{\om}=H_{0}+V_{\om}\quad\text{on}\quad L^{2}(\R^{d}),
\end{equation}
where $H_{0}=-\De$ is the negative Laplacian and $V_{\om}$ is the random sparse potential of the form $\sum_{i\in\Z^{d}}q_{i}(\om)\xi_{i}(\om)u(\cdot-i)$ satisfying the following assumptions
\begin{itemize}
\item[\rm(H1)] $\{\xi_{i}\}_{i\in\Z^{d}}$ are independent $\{0,1\}$ Bernoulli variables on some probability space $(\Om,\BB,\P,\E)$. Let $p_{i}=\P\{\om\in\Om|\xi_{i}(\om)=1\}=\E\{\xi_{i}\}$ and suppose that
\begin{equation*}
p_{i}\sim(1+|i|)^{-\al},\quad\frac{d+1}{2}<\al\leq d.
\end{equation*}
Here and in the sequel, $p_{i}\sim(1+|i|)^{-\al}$ means there are universal constants $c,C>0$ such that $\frac{c}{(1+|i|)^{\al}}\leq p_{i}\leq\frac{C}{(1+|i|)^{\al}}$ for all $i\in\Z^{d}$.
\item[\rm(H2)] $\{q_{i}\}_{i\in\Z^{d}}$ are i.i.d nonnegative random variables on $(\Om,\BB,\P,\E)$ with finite mean, i.e., $\E\{q_{i}\}<\infty$, and are independent from $\{\xi_{i}\}_{i\in\Z^{d}}$. Also, we assume that for a.e. $\om\in\Om$, $\sup_{i\in\Z^{d}}q_{i}(\om)<\infty$.
\item[\rm(H3)] $u:\R^{d}\ra\R$ is a non-trivial, bounded and nonpositive function with  support contained in $C_{0}:=(-\frac{1}{2},\frac{1}{2})^{d}$.
\end{itemize}

Spectral and dynamical aspects of the model \eqref{main-model} have been studied. In \cite{HK00}, Hundertmark and Kirsch proved the coexistence of wave operators and essential spectrum below zero if $p_{i}\sim(1+|i|)^{-\al}$ with $2<\al\leq d$ (while $\al>2$ ensures the existence of wave operator via Cook's method, $\al\leq d$ guarantees the existence of essential spectrum below zero). Clearly, this can be arranged for any dimension $d\geq3$ but not for $d=1$ or $d=2$. Under additional assumptions, say $p_{i}\sim(1+|i|)^{-\al}$ with $\al>\frac{d}{2}$, a Klaus-type theorem for the essential spectrum was also established. Later, Boutet de Monvel, Stollmann and Stolz proved in \cite{BSS05} the absence of a.c. spectrum below zero if $p_{i}\sim(1+|i|)^{-\al}$ with $\al>d-1$. This result, recalled in Proposition \ref{prop-ac-absence}, will enable us to prove the completeness of wave operators.

We point out that for Schr\"{o}dinger operators with random sparse potentials, completeness of wave operators are only known in the discrete case so far (see \cite{JP09,Po08}). Their proofs are based on the Jak\v{s}i\'{c}-Last criterion of completeness (see \cite{JL03}), which is true only for discrete models.

Consider the Schr\"{o}dinger operator $H_{\om}$ in \eqref{main-model} with assumptions $\rm(H1)$, $\rm(H2)$ and $\rm(H3)$. To apply the smooth method of Kato (see e.g. \cite{Ka65,RS78,Ya92}), by writing
$H_{\om}-H_{0}=-\sqrt{-V_{\om}}\sqrt{-V_{\om}}$, we need to show the local $H_{0}$-smoothness and local $H_{\om}$-smoothness of $\sqrt{-V_{\om}}$. To do so, we define for $\la>0$ and $\ep>0$ the modified free resolvent
\begin{equation*}
F_{\om}(\la+i\ep)=-\sqrt{-V_{\om}}(H_{0}-\la-i\ep)^{-1}\sqrt{-V_{\om}}
\end{equation*}
and the modified perturbed resolvent
\begin{equation*}
P_{\om}(\la+i\ep)=-\sqrt{-V_{\om}}(H_{\om}-\la-i\ep)^{-1}\sqrt{-V_{\om}}.
\end{equation*}
According to \cite[Theorem XIII.30]{RS78}, the uniform boundedness of $F_{\om}$ and $P_{\om}$ lead to the local smoothness. Considering this, we now state the first main result of the paper.

\begin{thm}\label{thm-main}
Let $d\geq2$. Suppose $\rm(H1)$, $\rm(H2)$ and $\rm(H3)$. Then, for a.e. $\om\in\Om$, the following statements hold:
\begin{itemize}
\item[\rm(i)] for any $0<a<b<\infty$, we have
\begin{equation*}
\begin{split}
&\sup_{\la\in[a,b],\ep\in(0,1]}\|F_{\om}(\la+i\ep)\|_{\LLL(L^{2}(\R^{d}))}<\infty\quad\text{and}\\
&\sup_{\la\in[a,b],\ep\in(0,1]}\|P_{\om}(\la+i\ep)\|_{\LLL(L^{2}(\R^{d}))}<\infty,
\end{split}
\end{equation*}
where $\LLL(L^{2}(\R^{d}))$ is the space of bounded linear operators on $L^{2}(\R^{d})$;

\item[\rm(ii)] the wave operators $\Om_{\pm}(H_{\om},H_{0})=s\mbox{-}\lim_{t\ra\pm\infty}e^{iH_{\om}t}e^{-iH_{0}t}$
exist and are complete. In particular, $\si_{ac}(H_{\om})=[0,\infty)$.
\end{itemize}
\end{thm}

As opposed to the results obtained by Hundertmark and Kirsch in \cite{HK00}, completeness of wave operators is established here at a price of making the decay rate $\al$ larger (with the same value in dimension $d=3$). But, in the present of essential spectrum below zero (i.e., $\al\leq d$), our results also hold in dimension $d=2$, which are previously unknown even for the existence of wave operators.

With the decay rate $\al$ satisfying $\frac{d+1}{2}<\al\leq d$, the Klaus-type theorem established in \cite{HK00} is as follows: for a.e. $\om\in\Om$, the essential spectrum below zero of $H_{\om}$ is given by $\ol{\cup_{\la\in\supp(\P_{0})}E(\la)}$, where $\P_{0}$ is the common distribution of $\{q_{i}\}_{i\in\Z^{d}}$ and $E(\la)$ is the set of all eigenvalues of $H_{0}+\la u$. Obviously, the set $\ol{\cup_{\la\in\supp(\P_{0})}E(\la)}$ could be very complicated, but this does not prevent us from proving the completeness of wave operators due to the result of Boutet de Monvel, Stollmann and Stolz (see \cite{BSS05}) as mentioned before. This situation is quite different from that of deterministic sparse potentials of the form $\sum_{i=1}^{\infty}u_{n}(\cdot-x_{n})$ as studied in \cite{Sh14}. For deterministic sparse potentials, except for the Klaus-type theorem, no more precise description of the essential spectrum below zero is available so far in general, and hence, a.c. spectrum below zero can not be excluded, which disables the proof of completeness of wave operators.

Finally, we genrealize Theorem \ref{thm-main}.  Consider the model
\begin{equation*}
\tilde{H}_{\om}=H_{0}+\tilde{V}_{\om}=H_{0}+\sum_{i\in\Z^{d}}\om_{i}u_{i}
\end{equation*}
under the assumptions
\begin{itemize}
\item[\rm(H4)] $\{\om_{i}\}_{i\in\Z^{d}}$ are independent random variables on some probability space $(\Om,\BB,\P,\E)$. Suppose for a.e. $\om\in\Om$, $\sup_{i\in\Z^{d}}|\om_{i}|<\infty$ and
\begin{equation*}
p_{i}=\E\{|\om_{i}|\}\sim(1+|i|)^{-\al},\quad\al>\frac{d+1}{2}.
\end{equation*}
\item[\rm(H5)] $\{u_{i}\}_{i\in\R^{d}}$ are real-valued functions satisfying $\sup_{i\in\Z^{d}}\|u_{i}\|_{\infty}<\infty$ and $\supp(u_{i})\subset i+B$ for all $i\in\Z^{d}$, where $B\subset\R^{d}$ is an open ball containing $0$.
\item[\rm(H6)] There are $m<0<M$ such that $\P\{\om\in\Om|\om_{i}\in[m,M]\}=1$ for all $i\in\Z^{d}$ and for all $\ep>0$,
\begin{equation*}
\P\{\om\in\Om||\om_{i}|\geq\ep\}\sim(1+|i|)^{-\al},\quad\al>d-1.
\end{equation*}
\end{itemize}
Then,
\begin{thm}\label{thm-main-4}
Let $d\geq2$. Suppose $\rm(H4)$ and $\rm(H5)$. Then,
\begin{itemize}
\item[\rm(i)] for a.e. $\om\in\Om$, there exist the strong limits
\begin{equation*}
\begin{split}
\Om_{\pm}(\tilde{H}_{\om},H_{0};I)&=s\mbox{-}\lim_{t\ra\pm\infty}e^{i\tilde{H}_{\om}t}e^{-iH_{0}t}\chi_{I}(H_{0}),\\
\Om_{\pm}(H_{0},\tilde{H}_{\om};I)&=s\mbox{-}\lim_{t\ra\pm\infty}e^{iH_{0}t}e^{-i\tilde{H}_{\om}t}\chi_{I}(\tilde{H}_{0})
\end{split}
\end{equation*}
for any interval $I\subset(0,\infty)$. In particular, for a.e. $\om\in\Om$, there exist the wave operators $\Om_{\pm}(\tilde{H}_{\om},H_{0})$ and the strong limits
\begin{equation*}
s\mbox{-}\lim_{t\ra\pm\infty}e^{iH_{0}t}e^{-i\tilde{H}_{\om}t}P_{ac}(\tilde{H}_{\om})\chi_{[0,\infty)}(\tilde{H}_{\om}),
\end{equation*}
where $P_{ac}(\tilde{H}_{\om})$ is the projection onto the a.c. subspace of $\tilde{H}_{\om}$.

\item[\rm(ii)] assume, in addition, $\rm(H6)$, then the wave operators $\Om_{\pm}(\tilde{H}_{\om},H_{0})$ exist and are complete.
\end{itemize}
\end{thm}

The rest of the paper is organized as follows. Section \ref{sec-uniform-boundedness} is the main part of the paper.  In Subsection \ref{sec-free}, we prove the uniform boundedness of modified free resolvents. In Subsection \ref{sec-perturbed}, we show the the uniform boundedness of modified perturbed resolvents. In Subsection \ref{sec-proof-main-results}, we prove Theorem \ref{thm-main}. In Subsection \ref{sec-proof-of-main-thm-4}, we prove Theorem \ref{thm-main-4}.

Throughout the paper, we use the following notations: for two nonnegative numbers $a$ and $b$, by writing $a\lesssim b$ we mean there's some $C>0$ such that $a\leq Cb$; $\lesssim_{c_{1},c_{2},\dots}$ is used to show the dependence on $c_{1},c_{2},\dots$; $a\sim b$ means $a\lesssim b$ and $b\lesssim a$; $\sim_{c_{1},c_{2},\dots}$ is used to show the dependence on $c_{1},c_{2},\dots$; $\chi_{B}$ is the characteristic function of $B\subset\R^{d}$; the norm and the inner product on $L^{2}=L^{2}(\R^{d})$ are denoted by $\|\cdot\|$ and $\lan\cdot,\cdot\ran$; for two Banach spaces $X$ and $Y$, we denote by $\LLL(X,Y)$ the space of all bounded linear operators from $X$ to $Y$ and the norm on $\LLL(X,Y)$ is denoted by $\|\cdot\|=\|\cdot\|_{\LLL(X,Y)}$; if $X=Y$, we write $\LLL(X)=\LLL(X,Y)$.


\section{Uniform boundedness of modified resolvents}\label{sec-uniform-boundedness}

For $z\in\C\bs\R$, we set $R_{0}(z)=(H_{0}-z)^{-1}$ and $R_{_{\om}}(z)=(H_{\om}-z)^{-1}$, the free resolvent and the perturbed resolvent. It is well-known (see e.g. \cite[page 288]{GS77}) that $R_{0}(z)$ is an integral operator with an explicit integral kernel
\begin{equation}\label{int-kernel-free}
k_{z}(x,y)=k_{z}(|x-y|)=c_{d}\bigg(\frac{\sqrt{z}}{|x-y|}\bigg)^{(d-2)/2}K_{(d-2)/2}(-i\sqrt{z}|x-y|),
\end{equation}
where $c_{d}$ is a constant depends only on $d$,  $\sqrt{z}$ satisfies $\Im\sqrt{z}>0$ and $K_{\bullet}$ is the Bessel potential.

We modify both the free resolvent $R_{0}(z)$ and the perturbed resolvent $R_{\om}(z)$ by defining the following operators: for $z\in\C\bs\R$,
\begin{equation*}
\begin{split}
F_{\om}(z)&=-\sqrt{-V_{\om}}R_{0}(z)\sqrt{-V_{\om}},\\
P_{\om}(z)&=-\sqrt{-V_{\om}}R_{\om}(z)\sqrt{-V_{\om}}.
\end{split}
\end{equation*}
By the resolvent identity $R_{\om}(z)-R_{0}(z)=-R_{\om}(z)V_{\om}R_{0}(z)$, we check
\begin{equation*}
P_{\om}(z)(1+F_{\om}(z))=F_{\om}(z).
\end{equation*}
Therefore, Theorem \ref{thm-main}$\rm(i)$ may be established if enough information about $F_{\om}(z)$ can be acquired.

In the sequel, we fix $0<a<b<\infty$ and set
\begin{equation*}
\begin{split}
\SS&=\Big\{z=\la+i\ep\in\C_{+}\Big|\la\in[a,b]\,\,\text{and}\,\,\ep\in(0,1]\Big\},\\
\ol{\SS}&=\Big\{z=\la+i\ep\in\C_{+}\Big|\la\in[a,b]\,\,\text{and}\,\,\ep\in[0,1]\Big\}.
\end{split}
\end{equation*}
In the following two subsections, we prove for $z\in\SS$ the uniform boundedness of $F_{\om}(z)$ and of $(1+F_{\om}(z))^{-1}$, which then allows us to establish the uniform boundedness of $P_{\om}(z)$.

In the rest of this section, $\rm(H1)$, $\rm(H2)$ and $\rm(H3)$ are assumed.

\subsection{Uniform boundedness of $F_{\om}(z)$}\label{sec-free}

Note $F_{\om}(z)$ is an integral operator with integral kernel
\begin{equation}\label{integral-kernel}
k_{F_{\om}(z)}(x,y)=-\sqrt{-V_{\om}(x)}k_{z}(x,y)\sqrt{-V_{\om}(y)},
\end{equation}
where $k_{z}(x,y)$ is the integral kernel of $R_{0}(z)$ given in \eqref{int-kernel-free}. For $\la\in[a,b]$ we set
\begin{equation*}
k_{\la+i0}(x,y)=\lim_{z\in\SS,z\ra\la}k_{z}(x,y).
\end{equation*}
Then, we define $F_{\om}(\la+i0)$ to be the integral operator with integral kernel
\begin{equation}\label{integral-kernel-1}
k_{F_{\om}(\la+i0)}(x,y)=-\sqrt{-V_{\om}(x)}k_{\la+i0}(x,y)\sqrt{-V_{\om}(y)}.
\end{equation}

We recall the following property of $k_{z}(x,y)$ (see e.g. \cite[page 418]{IS06}), which comes from the standard estimates of Bessel potentials.
\begin{lem}\label{estimate-of-kernel}
There's a constant $C=C(a,b,\de)$ such that for any $z\in\ol{\SS}$
\begin{equation*}
\begin{split}
|k_{z}(x,y)|=|k_{z}(|x-y|)|\leq \left\{ \begin{aligned}
C\ln\big(2|x-y|^{-1}\big),&\quad\text{if}\,\,|x-y|\leq\de\,\,\text{and}\,\,d=2,\\
C|x-y|^{-(d-2)},&\quad\text{if}\,\,|x-y|\leq\de\,\,\text{and}\,\,d\geq3,\\
C|x-y|^{-(d-1)/2},&\quad\text{if}\,\,|x-y|>\de\,\,\text{and}\,\,d\geq2.
\end{aligned} \right.
\end{split}
\end{equation*}
\end{lem}

We now prove the uniform boundedness of $F_{\om}(z)$ for $z\in\ol{\SS}$.

\begin{thm}\label{thm-free-res}
For a.e. $\om\in\Om$, $F_{\om}(z)$ is uniformly bounded on $\ol{\SS}$, that is,
\begin{equation*}
\sup_{z\in\ol{\SS}}\|F_{\om}(z)\|<\infty.
\end{equation*}
\end{thm}
\begin{proof}
Since $\supp(u)\subset C_{0}$, for $\om\in\Om$, we have the diagonal and off-diagonal decomposition
\begin{equation}\label{decomp-1}
F_{\om}(z)=I_{\om,z}+II_{\om,z},
\end{equation}
where $I_{\om,z}$ and $II_{\om,z}$ are integral operators with integral kernels
\begin{equation}\label{integral-kernels}
\begin{split}
I_{\om,z}(x,y)&=-\sum_{i\in\Z^{d}}q_{i}(\om)\xi_{i}(\om)\sqrt{-u(x-i)}k_{z}(x,y)\sqrt{-u(y-i)} \quad\text{and}\\
II_{\om,z}(x,y)&=-\sum_{i\neq j}\sqrt{q_{i}(\om)}\sqrt{q_{j}(\om)}\sqrt{\xi_{i}(\om)}\sqrt{\xi_{j}(\om)}\sqrt{-u(x-i)}k_{z}(x,y)\sqrt{-u(y-j)},
\end{split}
\end{equation}
respectively.

For $I_{\om,z}$, we claim that
\begin{equation}\label{estimate-1}
\sup_{z\in\ol{\SS}}\|I_{\om,z}\|<\infty\quad\text{for a.e.}\,\,\om\in\Om.
\end{equation}
By $\rm(H2)$, there is a measurable set $\Om_{1}$ of full probability such that
\begin{equation}\label{assumption-2}
\sup_{i\in\Z^{d}}q_{i}(\om)<\infty\quad\text{for all}\,\,\om\in\Om_{1}.
\end{equation}
Since for $i\in\Z^{d}$, $\xi_{i}$ is $0$-$1$ Bernoulli distributed, there is a measurable set $\Om_{i}$ of full probability such that $\xi_{i}(\om)=0$ or $1$ for all $\om\in\Om_{i}$. Setting $\Om_{2}=\cap_{i\in\Z^{d}}\Om_{i}$, we have $\P\{\Om_{2}\}=1$ and
\begin{equation}\label{assumption-1}
\sup_{i\in\Z^{d}}\xi_{i}(\om)\leq1\quad\text{for all}\,\,\om\in\Om_{2}.
\end{equation}
The claim \eqref{estimate-1} then follows if we show that
\begin{equation}\label{estimate-1-1}
\sup_{z\in\ol{\SS}}\|I_{\om,z}\|<\infty\quad\text{for all}\,\,\om\in\Om_{1}\cap\Om_{2}.
\end{equation}
To show \eqref{estimate-1-1}, we fix any $\om\in\Om_{1}\cap\Om_{2}$. Considering Lemma \ref{estimate-of-kernel}, we distinguish between $d=2$ and $d\geq3$. If $d\geq3$, we obtain from Lemma \ref{estimate-of-kernel}, the boundedness of $u$, \eqref{assumption-2} and \eqref{assumption-1} that for any $\phi,\psi\in L^{2}$
\begin{equation*}
\begin{split}
|\lan\phi,I_{\om,z}\psi\ran|&\lesssim\sum_{i\in\Z^{d}}\iint_{C_{i}\times C_{i}}\frac{|\phi(x)||\psi(y)|}{|x-y|^{d-2}}dxdy\\
&\leq\sum_{i\in\Z^{d}}\bigg(\iint_{C_{i}\times C_{i}}\frac{|\phi(x)|^{2}}{|x-y|^{d-2}}dxdy\bigg)^{1/2}\bigg(\iint_{C_{i}\times C_{i}}\frac{|\psi(x)|^{2}}{|x-y|^{d-2}}dxdy\bigg)^{1/2},
\end{split}
\end{equation*}
where $C_{i}=i+C_{0}$ for $i\in\Z^{d}$. Using the inequality $\chi_{C_{i}}(x)\chi_{C_{i}}(y)\leq\chi_{C_{i}-C_{i}}(x-y)$, we find
\begin{equation*}
\begin{split}
\iint_{C_{i}\times C_{i}}\frac{|\phi(x)|^{2}}{|x-y|^{d-2}}dxdy&=\iint_{\R^{d}\times\R^{d}}|\chi_{C_{i}}(x)\phi(x)|^{2}\frac{\chi_{C_{i}}(x)\chi_{C_{i}}(y)}{|x-y|^{d-2}}dxdy\\
&\leq\iint_{\R^{d}\times\R^{d}}|\chi_{C_{i}}(x)\phi(x)|^{2}\frac{\chi_{C_{i}-C_{i}}(x-y)}{|x-y|^{d-2}}dxdy\\
&=\int_{\R^{d}}|\chi_{C_{i}}(x)\phi(x)|^{2}\bigg(\int_{\R^{d}}\frac{\chi_{C_{i}-C_{i}}(x-y)}{|x-y|^{d-2}}dy\bigg)dx\\
&=\|\chi_{C_{i}}\phi\|^{2}\int_{C_{i}-C_{i}}\frac{1}{|x|^{d-2}}dx\\
&\lesssim\|\chi_{C_{i}}\phi\|^{2},
\end{split}
\end{equation*}
where we used the fact that the integrals $\int_{C_{i}-C_{i}}\frac{1}{|x|^{d-2}}dx$, $i\in\Z^{d}$ converge and are independent of $i$ in the last step.
Similarly,
\begin{equation*}
\iint_{C_{i}\times C_{i}}\frac{|\psi(x)|^{2}}{|x-y|^{d-2}}dxdy\lesssim\|\chi_{C_{i}}\psi\|^{2}.
\end{equation*}
It then follows that
\begin{equation*}
|\lan\phi,I_{\om,z}\psi\ran|\lesssim\sum_{i\in\Z^{d}}\|\chi_{C_{i}}\phi\|\|\chi_{C_{i}}\psi\|\leq\|\phi\|\|\psi\|.
\end{equation*}
This yields $\sup_{z\in\ol{\SS}}\|I_{\om,z}\|<\infty$. Therefore, \eqref{estimate-1-1} in the case $d\geq3$ holds. So does \eqref{estimate-1}.

The estimate \eqref{estimate-1} in the case $d=2$ can be treated similarly. The only difference is that we use the fact that the integrals $\int_{C_{i}-C_{i}}\big|\ln\frac{2}{|x|}\big|dx$, $i\in\Z^{d}$ converge and are independent of $i$. Thus, \eqref{estimate-1} also holds in the case $d=2$.

For $II_{\om,z}$, we claim that
\begin{equation}\label{estimate-2}
\E\Big\{\sup_{z\in\ol{\SS}}\|II_{\om,z}\|\Big\}<\infty.
\end{equation}
To show \eqref{estimate-2}, we control the operator-norm of $II_{\om,z}$ by the Hilbert-Schmidt norm of $II_{\om,z}$, that is, the $L^{2}(\R^{d}\times\R^{d})$-norm of the the integral kernel $II_{\om,z}(\cdot,\cdot)$. Moreover, using Jensen's inequality, we have
\begin{equation*}
\E\Big\{\sup_{z\in\ol{\SS}}\|II_{\om,z}\|\Big\}\leq\E\Big\{\sup_{z\in\ol{\SS}}\|II_{\om,z}(\cdot,\cdot)\|_{L^{2}(\R^{d}\times\R^{d})}\Big\}\leq\bigg(\E\Big\{\sup_{z\in\ol{\SS}}\|II_{\om,z}(\cdot,\cdot)\|^{2}_{L^{2}(\R^{d}\times\R^{d})}\Big\}\bigg)^{1/2}.
\end{equation*}
Thus, it suffices to show
\begin{equation}\label{estimate-HS}
\E\Big\{\sup_{z\in\ol{\SS}}\|II_{\om,z}(\cdot,\cdot)\|^{2}_{L^{2}(\R^{d}\times\R^{d})}\Big\}<\infty.
\end{equation}

We now show \eqref{estimate-HS}. Clearly, $\rm(H3)$ implies
\begin{equation}\label{upper-bound-2}
|II_{\om,z}(x,y)|\lesssim\sum_{i\neq j}\sqrt{q_{i}(\om)}\sqrt{q_{j}(\om)}\sqrt{\xi_{i}(\om)}\sqrt{\xi_{j}(\om)}\frac{\chi_{C_{i}}(x)\chi_{C_{j}}(y)}{|x-y|^{(d-1)/2}}.
\end{equation}
Then, Fubini's theorem and assumptions $\rm(H1)$, $\rm(H2)$ ensure
\begin{equation*}
\begin{split}
&\E\bigg\{\sup_{z\in\ol{\SS}}\iint_{\R^{d}\times\R^{d}}|II_{\om,z}(x,y)|^{2}dxdy\bigg\}\\
&\quad\quad\lesssim\E\bigg\{\iint_{\R^{d}\times\R^{d}}\bigg(\sum_{i\neq j}q_{i}(\om)q_{j}(\om)\xi_{i}(\om)\xi_{j}(\om)\frac{\chi_{C_{i}}(x)\chi_{C_{j}}(y)}{|x-y|^{d-1}}\bigg)dxdy\bigg\}\\
&\quad\quad\lesssim\iint_{\R^{d}\times\R^{d}}\bigg(\sum_{i\neq j}p_{i}p_{j}\frac{\chi_{C_{i}}(x)\chi_{C_{j}}(y)}{|x-y|^{d-1}}\bigg)dxdy\\
&\quad\quad\leq\iint_{\R^{d}\times\R^{d}}\frac{f(x)f(y)}{|x-y|^{d-1}}dxdy,
\end{split}
\end{equation*}
where $f(x)=\sum_{i\in\Z^{d}}p_{i}\chi_{C_{i}}(x)$. By means of the Hardy-Littlewood-Sobolev inequality, recalled in Proposition \ref{hls} below, with $\frac{d+1}{2d}+\frac{d-1}{d}+\frac{d+1}{2d}=2$ (this requires $d\geq2$), we find
\begin{equation*}
\iint_{\R^{d}\times\R^{d}}\frac{f(x)f(y)}{|x-y|^{d-1}}dxdy\leq C_{d}\|f\|_{\frac{2d}{d+1}}^{2}
\end{equation*}
for some $C_{d}>0$ depending only on $d$. It then follows that
\begin{equation}\label{HS}
\E\bigg\{\sup_{z\in\ol{\SS}}\iint_{\R^{d}\times\R^{d}}|II_{\om,z}(x,y)|^{2}dxdy\bigg\}\lesssim\|f\|_{\frac{2d}{d+1}}^{2}=\bigg(\sum_{i\in\Z^{d}}p_{i}^{\frac{2d}{d+1}}\bigg)^{\frac{d+1}{d}}<\infty,
\end{equation}
since $p_{i}\sim\frac{1}{(1+|i|)^{\al}}$ with $\al>\frac{d+1}{2}$ by $\rm(H1)$. This establishes \eqref{estimate-HS}, and hence, \eqref{estimate-2} follows.

Combining \eqref{decomp-1}, \eqref{estimate-1} and \eqref{estimate-2}, we finish the proof of the theorem.
\end{proof}

In the proof of Theorem \ref{thm-free-res}, we used the following Hardy-Littlewood-Sobolev inequality (see e.g. \cite{LL01}).

\begin{prop}\label{hls}
Let $p,r>1$ and $0<\la<d$ with $\frac{1}{p}+\frac{\la}{d}+\frac{1}{r}=2$. Let $f\in L^{p}(\R^{d})$ and $g\in L^{r}(\R^{d})$. Then there exists $C=C(d,\la,p)>0$ such that
\begin{equation*}
\bigg|\iint_{\R^{d}\times\R^{d}}\frac{f(x)g(x)}{|x-y|^{\la}}dxdy\bigg|\leq C\|f\|_{p}\|h\|_{r}.
\end{equation*}
\end{prop}

The proof of \eqref{HS} ensures the following:
\begin{equation*}
\E\bigg\{\iint_{\R^{d}\times\R^{d}}\bigg(\sum_{i\neq j}q_{i}(\om)q_{j}(\om)\xi_{i}(\om)\xi_{j}(\om)\frac{\chi_{C_{i}}(x)\chi_{C_{j}}(y)}{|x-y|^{d-1}}\bigg)dxdy\bigg\}<\infty,
\end{equation*}
which implies that there's some $\Om_{3}\subset\Om$ of full probability such that if $\om\in\Om_{3}$ then
\begin{equation}\label{integral-2}
\iint_{\R^{d}\times\R^{d}}\bigg(\sum_{i\neq j}q_{i}(\om)q_{j}(\om)\xi_{i}(\om)\xi_{j}(\om)\frac{\chi_{C_{i}}(x)\chi_{C_{j}}(y)}{|x-y|^{d-1}}\bigg)dxdy<\infty.
\end{equation}
We will need the following
\begin{cor}\label{cor-continuity-free}
For any $\om\in\Om_{3}$, $F_{\om}(z)$ is continuous on $\ol{\SS}$.
\end{cor}
\begin{proof}
Fix any $\om\in\Om_{3}$ and $z_{0}\in\ol{\SS}$. For $z\in\ol{\SS}$, we write
\begin{equation*}
F_{\om}(z)-F_{\om}(z_{0})=I_{\om,z}-I_{\om,z_{0}}+II_{\om,z}-II_{\om,z_{0}}.
\end{equation*}
Then, $I_{\om,z}-I_{\om,z_{0}}$ and $II_{\om,z}-II_{\om,z_{0}}$ are integral operators with integral kernels $I_{\om,z}(x,y)-I_{\om,z_{0}}(x,y)$ and $II_{\om,z}(x,y)-II_{\om,z_{0}}(x,y)$, respectively.

As in the proof of Theorem \ref{thm-free-res}, we have
\begin{equation*}
\|I_{\om,z}-I_{\om,z_{0}}\|\lesssim\int_{C_{0}-C_{0}}|k_{z}(|x|)-k_{z_{0}}(|x|)|dx.
\end{equation*}
Dominated convergence theorem then implies that $\lim_{z\ra z_{0}}I_{\om,z}=I_{\om,z_{0}}$ in $\LLL(L^{2})$. For $II_{\om,z}-II_{\om,z_{0}}$, we have
\begin{equation*}
\|II_{\om,z}-II_{\om,z_{0}}\|^{2}\leq\iint_{\R^{d}\times\R^{d}}|II_{\om,z}(x,y)-II_{\om,z_{0}}(x,y)|^{2}dxdy.
\end{equation*}
It then follows from \eqref{upper-bound-2}, \eqref{integral-2} and dominated convergence theorem that $\lim_{z\ra z_{0}}II_{\om,z}=II_{\om,z_{0}}$ in $\LLL(L^{2})$. This completes the proof.
\end{proof}

We end this section with the following remark.

\begin{rem}
A sufficient condition for the integrability and the almost sure uniform boundedness of $\{q_{i}\}_{i\in\Z^{d}}$ in $\rm(H2)$ is that $\{q_{i}\}_{i\in\Z^{d}}$ are bounded, that is, there's $c>0$ such that $\P\{\om\in\Om|q_{i}(\om)\leq c\}=1$. Moreover, if $\{q_{i}\}_{i\in\Z^{d}}$ are bounded, then following the proof of \eqref{estimate-1}, we easily check $\E\big\{\sup_{z\in\ol{\SS}}\|I_{\om,z}\|\big\}<\infty$,
hence,
\begin{equation*}
\E\Big\{\sup_{z\in\ol{\SS}}\|F_{\om}(z)\|\Big\}<\infty.
\end{equation*}
This is stronger than the result in Theorem \ref{thm-free-res}.
\end{rem}

\subsection{Uniform boundedness of $P_{\om}(z)$}\label{sec-perturbed}

Recall that for $\om\in\Om$ and $z\in\SS$,
\begin{equation*}
P_{\om}(z)=-\sqrt{-V_{\om}}R_{\om}(z)\sqrt{-V_{\om}}
\end{equation*}
and it satisfies
\begin{equation}\label{res-id}
P_{\om}(z)(1+F_{\om}(z))=F_{\om}(z),
\end{equation}
where $F_{\om}(z)=-\sqrt{-V_{\om}}R_{0}(z)\sqrt{-V_{\om}}$. We prove
\begin{thm}\label{thm-perturb-res}
For a.e. $\om\in\Om$, $P_{\om}(z)$ is uniformly bounded on $\SS$, that is,
\begin{equation*}
\sup_{z\in\SS}\|P_{\om}(z)\|<\infty.
\end{equation*}
\end{thm}

Considering Theorem \ref{thm-free-res}, Corollary \ref{cor-continuity-free} and the operator equation \eqref{res-id}, to prove Theorem \ref{thm-perturb-res}, it suffices to prove for a.e. $\om\in\Om$, the invertibility of $1+F_{\om}(z)$ for each $z\in\SS$ and the uniform boundedness of their inverses on $\SS$. As in \cite{Sh14}, this can be done through the following three steps:
\begin{itemize}
\item[\rm(i)] for a.e. $\om\in\Om$, $(1+F_{\om}(z))^{-1}$ exists for $z\in\SS$;
\item[\rm(ii)] for a.e. $\om\in\Om$, $(1+F_{\om}(\la+i0))^{-1}$ exists for $\la\in[a,b]$;
\item[\rm(iii)] for a.e. $\om\in\Om$, $(1+F_{\om}(z))^{-1}$ is continuous on the compact set $\ol{\SS}=\SS\cup\{\la+i0|\la\in[a,b]\}$.
\end{itemize}

We first establish the invertibility of $1+F_{\om}(z)$ for $z\in\SS$.

\begin{lem}\label{invertible-away-from-real}
Let $\om\in\Om_{3}$. For each $z\in\SS$, the operator $1+F_{\om}(z)$ is boundedly invertible.
\end{lem}
\begin{proof}
Fix any $z\in\SS$. We first show that $-1$ is neither an eigenvalue nor in the residue spectrum of $F_{\om}(z)$. Clearly, it suffices to show that $1+F_{\om}(z)$ is one-to-one and has dense range.  To do so, let $\phi\in L^{2}$ be such that $(1+F_{\om}(z))\phi=0$, that is,
\begin{equation}\label{an-equality-eigenfun}
\phi=\sqrt{-V_{\om}}R_{0}(z)\sqrt{-V_{\om}}\phi.
\end{equation}
Writing
\begin{equation*}
\sqrt{-V_{\om}}R_{0}(z)\sqrt{-V_{\om}}=\Re(\sqrt{-V_{\om}}R_{0}(z)\sqrt{-V_{\om}})+i\Im(\sqrt{-V_{\om}}R_{0}(z)\sqrt{-V_{\om}}), 
\end{equation*}
we conclude from
\begin{equation*}
\|\phi\|^{2}=\lan\phi,\Re(\sqrt{-V_{\om}}R_{0}(z)\sqrt{-V_{\om}})\phi\ran+i\lan\phi,\Im(\sqrt{-V_{\om}}R_{0}(z)\sqrt{-V_{\om}})\phi\ran
\end{equation*}
that $\lan\phi,\Im(\sqrt{-V_{\om}}R_{0}(z)\sqrt{-V_{\om}})\phi\ran=0$. Since
\begin{equation*}
\Im(\sqrt{-V_{\om}}R_{0}(z)\sqrt{-V_{\om}})=\sqrt{-V_{\om}}\Im R_{0}(z)\sqrt{-V_{\om}}=\Im z\sqrt{-V_{\om}}R_{0}(\ol{z})R_{0}(z)\sqrt{-V_{\om}},
\end{equation*}
we obtain $\Im z\|R_{0}(z)\sqrt{-V_{\om}}\phi\|^{2}=0$, yielding $R_{0}(z)\sqrt{-V_{\om}}\phi=0$, and so, $\phi=0$ by \eqref{an-equality-eigenfun}.

This shows that $1+F_{\om}(z)$ is one-to-one. A similar argument shows that $(1+F_{\om}(z))^{*}=1+F_{\om}(z)^{*}$ is also one-to-one. From the fact that $\ol{{\range}(1+F_{\om}(z))}\oplus\ker(1+F_{\om}(z)^{*})=L^{2}$, we conclude that $1+F_{\om}(z)$ has dense range. Hence, $1+F_{\om}(z)$ is one-to-one and has dense range. Hence, if $1+F_{\om}(z)$ is not boundedly invertible, then $(1+F_{\om}(z))^{-1}$ is densely defined and unbounded.

Now, we show that $1+F_{\om}(z)$ is boundedly invertible. For contradiction, we assume that $1+F_{\om}(z)$ is not boundedly invertible, that is, $-1\in\si(F_{\om}(z))$. Then, the above analysis says that there exists $\{\phi_{n}\}_{n\in\N}\subset L^{2}$ such that $\|\phi_{n}\|=1$ for all $n$ and $\|(1+F_{\om}(z))\phi_{n}\|\ra0$ as $n\ra\infty$. Define
$\psi_{n}=R_{0}(z)\sqrt{-V_{\om}}\phi_{n}$. We claim that there's some $C>0$ such that $\inf_{n\in\N}\|\psi_{n}\|\geq C$. In fact, if this is not true, then we can find some subsequence $\{\phi_{n_{k}}\}_{k\in\N}$ such that $\|\psi_{n_{k}}\|=\|R_{0}(z)\sqrt{-V_{\om}}\phi_{n_{k}}\|\ra0$ as $k\ra\infty$, which leads to the following contradiction:
\begin{equation*}
\begin{split}
1=\|\phi_{n_{k}}\|&\leq\|(1+F_{\om}(z))\phi_{n_{k}}\|+\|F_{\om}(z)\phi_{n_{k}}\|\\
&\leq\|(1+F_{\om}(z))\phi_{n_{k}}\|+\|\sqrt{-V_{\om}}\|_{\infty}\|R_{0}(z)\sqrt{-V_{\om}}\phi_{n_{k}}\|\\
&\ra0\quad\text{as}\,\,k\ra\infty.
\end{split}
\end{equation*}

Also, setting $\vp_{n}=(1+F_{\om}(z))\phi_{n}$, i.e., $\phi_{n}=\vp_{n}-F_{\om}(z)\phi_{n}$, we have
\begin{equation*}
(H_{0}-z)\psi_{n}=\sqrt{-V_{\om}}\phi_{n}=\sqrt{-V_{\om}}\vp_{n}-\sqrt{-V_{\om}}F_{\om}(z)\phi_{n}=\sqrt{-V_{\om}}\vp_{n}-V_{\om}\psi_{n},
\end{equation*}
that is, $(H_{0}+V_{\om}-z)\psi_{n}=\sqrt{-V_{\om}}\vp_{n}$, or $(H_{0}+V_{\om}-z)\frac{\psi_{n}}{\|\psi_{n}\|}=\frac{\sqrt{-V_{\om}}\vp_{n}}{\|\psi_{n}\|}$.
Since $\inf_{n\in\N}\|\psi_{n}\|\geq C$, we have $\|\frac{\sqrt{-V_{\om}}\vp_{n}}{\|\psi_{n}\|}\|\leq\frac{\|\sqrt{-V_{\om}}\|_{\infty}\|\vp_{n}\|}{C}\ra0$ as $n\ra\infty$. The bounded invertibility of $H_{0}+V-z$ then ensures that
\begin{equation*}
\frac{\psi_{n}}{\|\psi_{n}\|}=(H_{0}+V_{\om}-z)^{-1}\frac{\sqrt{-V_{\om}}\vp_{n}}{\|\psi_{n}\|}\ra0\quad\text{as}\,\,n\ra\infty,
\end{equation*}
which leads to a contradiction. Consequently, $-1\in\rho(F_{\om}(z))$, that is, $1+F_{\om}(z)$ is boundedly invertible.
\end{proof}

Next, we prove the invertibility of $1+F_{\om}(\la+i0)$ for $\la\in[a,b]$.

\begin{lem}\label{invertible-on-the-real}
Let $\om\in\Om_{3}$. For each $\la\in[a,b]$, the operator $1+F_{\om}(\la+i0)$ is boundedly invertible.
\end{lem}
\begin{proof}
Fix any $\la\in[a,b]$. We first claim that there exists $C>0$ such that
\begin{equation}\label{bound-below}
\inf_{\ep\in(0,1]}\|1+F_{\om}(\la+i\ep)\|\geq C.
\end{equation}
In fact, if this is not the case, Corollary \ref{cor-continuity-free} and Lemma \ref{invertible-away-from-real} then imply that there exists $\{\ep_{n}\}_{n\in\N}$ such that $\ep_{n}\ra0$ as $n\ra\infty$ and
\begin{equation*}
\|1+F_{\om}(\la+i\ep_{n})\|\ra0\quad\text{as}\,\,n\ra\infty.
\end{equation*}
This means that $\{F_{\om}(\la+i\ep_{n})\}_{n\in\N}$ converges in norm to the operator $-I$, where $I$ is the identity on $L^{2}$. Since $\{F_{\om}(\la+i\ep_{n})\}_{n\in\N}$ also converges in norm to the operator $F_{\om}(\la+i0)$ by Corollary \ref{cor-continuity-free}, we have $F_{\om}(\la+i0)=-I$. But, clearly, this is not the case since $F_{\om}(\la+i0)$ is the integral operator with the integral kernel $-\sqrt{-V_{\om}(x)}k_{0,\la}(x,y)\sqrt{-V_{\om}(y)}$. Hence, \eqref{bound-below} is true. By \eqref{bound-below} and Lemma \ref{invertible-away-from-real}, we find
\begin{equation}\label{bound-above}
\sup_{\ep\in(0,1]}\|(1+F_{\om}(\la+i\ep))^{-1}\|\leq\frac{1}{C}.
\end{equation}

Next, we prove the bounded invertibility of $1+F_{\om}(\la+i0)$.  By Corollary \ref{cor-continuity-free}, we can find some $\ep_{0}\in(0,1]$ such that $\|F_{\om}(\la+i0)-F_{\om}(\la+i\ep_{0})\|<C$, where $C$ is the same as that in \eqref{bound-above}. It then follows from \eqref{bound-above} that
\begin{equation*}
\|F_{\om}(\la+i0)-F_{\om}(\la+i\ep_{0})\|<C\leq\frac{1}{\sup_{\ep\in(0,1]}\|(1+F_{\om}(\la+i\ep))^{-1}\|}\leq\frac{1}{\|(1+F_{\om}(\la+i\ep_{0}))^{-1}\|}.
\end{equation*}
Stability of bounded invertibility (see e.g. \cite{La02}) then implies the bounded invertibility of
\begin{equation*}
1+F_{\om}(\la+i0)=1+F_{\om}(\la+i\ep_{0})+[F_{\om}(\la+i0)-F_{\om}(\la+i\ep_{0})].
\end{equation*}
\end{proof}

Lemma \ref{invertible-away-from-real} and Lemma \ref{invertible-on-the-real} together say that for $\om\in\Om_{3}$, $1+F_{\om}(z)$ is boundedly invertible for each $z\in\ol{\SS}$. In the next result, we prove the uniform boundedness of their inverses.

\begin{lem}\label{uniform-bounded-inverse}
Let $\om\in\Om_{3}$. Then, $(1+F_{\om}(z))^{-1}$ is continuous on $\ol{\SS}$. In particular, there holds
\begin{equation*}
\sup_{z\in\ol{\SS}}\|(1+F_{\om}(z))^{-1}\|<\infty.
\end{equation*}
\end{lem}
\begin{proof}
Due to the compactness of $\ol{\SS}$, it suffices to show the continuity of $(1+F_{\om}(z))^{-1}$ on $\ol{\SS}$. For this purpose, we fix any $z_{0}\in\ol{\SS}$. Then, for any $z\in\ol{\SS}$, we have the formal expansion
\begin{equation*}
\begin{split}
(1+F_{\om}(z))^{-1}=\bigg(1+\sum_{n=1}^{\infty}[(1+F_{\om}(z_{0}))^{-1}(F_{\om}(z_{0})-F_{\om}(z))]^{n}\bigg)(1+F_{\om}(z_{0}))^{-1}.
\end{split}
\end{equation*}
The above series converges if $\|(1+F_{\om}(z_{0}))^{-1}(F_{\om}(z_{0})-F_{\om}(z))\|<1$, which, by Corollary \ref{cor-continuity-free}, is true for all $z$ close to $z_{0}$. Thus, for any $z\in\SS$ close to $z_{0}$, we deduce
\begin{equation*}
\|(1+F_{\om}(z))^{-1}-(1+F_{\om}(z_{0}))^{-1}\|\leq\frac{\|(1+F_{\om}(z_{0}))^{-1}(F_{\om}(z_{0})-F_{\om}(z))\|}{1-\|(1+F_{\om}(z_{0}))^{-1}(F_{\om}(z_{0})-F_{\om}(z))\|}\|(1+F_{\om}(z_{0}))^{-1}\|\ra0
\end{equation*}
as $z\ra z_{0}$. This establishes the continuity of $(1+F_{\om}(z))^{-1}$ at $z_{0}$, and thus, $(1+F_{\om}(z))^{-1}$ is continuous on $\ol{\SS}$.
\end{proof}

Finally, we prove Theorem \ref{thm-perturb-res}.

\begin{proof}[Proof of Theorem \ref{thm-perturb-res}]
The result follows from \eqref{res-id}, Theorem \ref{thm-free-res}, Lemma \ref{invertible-away-from-real}, Lemma \ref{invertible-on-the-real} and Lemma \ref{uniform-bounded-inverse}. In fact, for a.e. $\om\in\Om$, $P_{\om}(z)$ can be continuously extended to $\ol{\SS}$. In particular, for a.e. $\om\in\Om$, there holds $\sup_{z\in\ol{\SS}}\|P_{\om}(z)\|<\infty$.
\end{proof}


\section{Proof of Main Results}

In this section, we proof Theorem \ref{thm-main} and Theorem  \ref{thm-main-4}.

\subsection{Proof of Theorem \ref{thm-main}}\label{sec-proof-main-results}

To apply Kato's smooth method, besides Theorem \ref{thm-free-res} and Theorem \ref{thm-perturb-res}, we also need the following result due to Boutet de Monvel, Stollmann and Stolz (see \cite[Theorem 3.1]{BSS05}).

\begin{prop}\label{prop-absence-ac}
Suppose $\rm(H1)$, $\rm(H2)$ and $\rm(H3)$. Then, $\si_{ac}(H_{\om})\cap(-\infty,0)=\emptyset$  for a.e. $\om\in\Om$.
\end{prop}

A more general version of Proposition \ref{prop-absence-ac} is stated in Proposition \ref{prop-ac-absence}.

We now prove Theorem \ref{thm-main}.

\begin{proof}[Proof of Theorem \ref{thm-main}]
By Theorem \ref{thm-free-res} and Theorem \ref{thm-perturb-res}, for any compact interval $I\subset(0,\infty)$ there exists a set $\Om_{I}\subset\Om$ of full probability such that for any $\om\in\Om_{I}$,  there hold
\begin{equation}\label{lap-proof}
\sup_{\la\in I,\ep\in(0,1]}\|F_{\om}(\la+i\ep)\|<\infty\quad\text{and}\quad\sup_{\la\in I,\ep\in(0,1]}\|P_{\om}(\la+i\ep)\|<\infty.
\end{equation}
Now, let $\{I_{n}\}_{n\in\N}$ be a sequence of compact intervals such that $I_{1}\subset I_{2}\subset\cdots\subset I_{n}\subset\cdots$ and $(0,\infty)=\cup_{n\in\N}I_{n}$. Denote by $\{\Om_{I_{n}}\}_{n\in\N}$ the corresponding sequence of sets of full probability. Set
\begin{equation*}
\Om_{*}=\bigcap_{n\in\N}\Om_{I_{n}}.
\end{equation*}
Then, $\P\{\Om_{*}\}=1$ and for any $\om\in\Om_{*}$, \eqref{lap-proof} with $I$ replaced by any $I_{n}$ is true. This clearly implies Theorem \ref{thm-main}$\rm(i)$.

We now fix any $\om\in\Om_{*}$. For the existence and completeness of local wave operators $s\mbox{-}\lim_{t\ra\pm\infty}e^{iH_{\om}t}e^{-iH_{0}t}\chi_{I}(H_{0})$ for some compact interval $I\subset(0,\infty)$, we invoke \cite[Theorem XIII.31]{RS78}. To do so, writing $H_{\om}-H_{0}=-\sqrt{-V_{\om}}\sqrt{-V_{\om}}$, since clearly $\sqrt{-V_{\om}}$ is both $H_{0}$-bounded and $H_{\om}$-bounded, we only need to show that $\sqrt{-V_{\om}}$ is both $H_{0}$-smooth and $H_{\om}$-smooth on $I$.

For the $H_{0}$-smoothness and $H$-smoothness of $\sqrt{-V_{\om}}$ on $I$, \cite[Theorem XIII.30]{RS78} says that it suffices to show that
\begin{equation*}
\begin{split}
\sup_{\la\in I,\ep\in(0,1]}\|\sqrt{-V_{\om}}R_{0}(\la+i\ep)\sqrt{-V_{\om}}\|<\infty,\\
\sup_{\la\in I,\ep\in(0,1]}\|\sqrt{-V_{\om}}R_{\om}(\la+i\ep)\sqrt{-V_{\om}}\|<\infty,
\end{split}
\end{equation*}
which are the statement of the first part of the theorem. Thus, we have shown the existence and completeness of local wave operators, that is, for any $\om\in\Om_{*}$, the strong limits
\begin{equation}\label{local-wave-op}
s\mbox{-}\lim_{t\ra\pm\infty}e^{iH_{\om}t}e^{-iH_{0}t}\chi_{I}(H_{0})
\end{equation}
\begin{equation}\label{local-inverse-wave-op}
s\mbox{-}\lim_{t\ra\pm\infty}e^{iH_{0}t}e^{-iH_{\om}t}\chi_{I}(H_{\om})
\end{equation}
exist for any compact interval $I\subset(0,\infty)$.

For the existence and completeness of wave operators, we first note that \eqref{local-wave-op} implies the existence of wave operators, i.e.,
\begin{equation*}
s\mbox{-}\lim_{t\ra\pm\infty}e^{iH_{\om}t}e^{-iH_{0}t}\,\,\text{exist for all}\,\,\om\in\Om_{*}
\end{equation*}
and \eqref{local-inverse-wave-op} implies
\begin{equation}\label{inverse-wave-op-1}
s\mbox{-}\lim_{t\ra\pm\infty}e^{iH_{0}t}e^{-iH_{\om}t}P_{ac}(H_{\om})\chi_{[0,\infty)}(H_{\om})\,\,\text{exist for all}\,\,\om\in\Om_{*},
\end{equation}
where $P_{ac}(H_{\om})$ is the projection onto the a.c. subspace of $H_{\om}$. To prove the existence of inverse wave operators, by Proposition \ref{prop-absence-ac}, we let $\Om_{**}$ be the set of full probability such that $\si_{ac}(H_{\om})\cap(-\infty,0)=\emptyset$  for all $\om\in\Om_{**}$.
In particular,
\begin{equation}\label{absence-of-ac}
P_{ac}(H_{\om})\chi_{(-\infty,0)}(H_{\om})=0\,\,\text{for all}\,\,\om\in\Om_{**}.
\end{equation}
We then conclude from \eqref{inverse-wave-op-1} and \eqref{absence-of-ac} that the inverse wave operators
\begin{equation*}
s\mbox{-}\lim_{t\ra\pm\infty}e^{iH_{0}t}e^{-iH_{\om}t}P_{ac}(H_{\om})\,\,\text{exist for all}\,\,\om\in\Om_{*}\cap\Om_{**}.
\end{equation*}
Consequently, for any $\om\in\Om_{*}\cap\Om_{**}$, the wave operators $s\mbox{-}\lim_{t\ra\pm\infty}e^{iH_{\om}t}e^{-iH_{0}t}$ exist and are complete.
\end{proof}


\subsection{Proof of Theorem \ref{thm-main-4}}\label{sec-proof-of-main-thm-4}

Theorem \ref{thm-main-4} follows from the arguments as in the proof of Theorem \ref{thm-main}. We sketch the proof by pointing out the differences, which are mainly caused by the fact that ${\supp}(u_{i})$, ${i\in\Z^{d}}$ are no longer pairwise disjoint. For notational simplicity, we will omit ``for a.e. $\om\in\Om$" and focus on certain $\tilde{H}_{\om}$.

For $\la>0$ and $\ep>0$, define
\begin{equation*}
\begin{split}
\tilde{F}_{\om}(\la+i\ep)&=|\tilde{V}_{\om}|^{1/2}(H_{0}-\la-i\ep)^{-1}\tilde{V}_{\om}^{1/2},\\
\tilde{P}_{\om}(\la+i\ep)&=|\tilde{V}_{\om}|^{1/2}(\tilde{H}_{\om}-\la-i\ep)^{-1}\tilde{V}_{\om}^{1/2},
\end{split}
\end{equation*}
where $\tilde{V}_{\om}^{1/2}={\sgn}(\tilde{V}_{\om})|\tilde{V}_{\om}|^{1/2}$.

We first prove Theorem \ref{thm-main-4}$\rm(i)$. We claim that
\begin{equation}\label{claim-10}
\sup_{z\in\ol{\SS}}\|\tilde{F}_{\om}(z)\|<\infty,
\end{equation}
\begin{equation}\label{claim-11}
\tilde{F}_{\om}(z)\,\,\text{is continuous on}\,\,\ol{\SS}.
\end{equation}
Note that \eqref{claim-11} is a simple consequence of the Lebesgue dominated convergence as in the proof of Corollary \ref{cor-continuity-free}. Moreover, once \eqref{claim-10} and \eqref{claim-11} are established, we can readily check $\sup_{z\in\SS}\|\tilde{P}_{\om}(z)\|<\infty$ as in the proof of Theorem \ref{thm-perturb-res}, which then leads to the results. Thus, it suffices to show \eqref{claim-10}.

We only prove \eqref{claim-10} in the case $d\geq3$; the $d=2$ case is similar. By $\rm(H5)$, we may assume w.l.o.g that for $u_{i}\subset C_{R}(i)$ for some odd $R\geq3$, where $C_{r}(i)\subset\R^{d}$ is the open cube centered at $i$ with side length $r>0$. Thus,
\begin{equation}\label{disjoint-supp}
{\rm dist}(\supp(u_{i}),\supp(u_{j}))>0\,\,\text{if}\,\,|i-j|_{\infty}\geq R,
\end{equation}
where $|i-j|_{\infty}=\max_{k=1,\dots,d}|i_{k}-j_{k}|$. Then, for $\phi,\psi\in L^{2}$
\begin{equation*}
\begin{split}
|\lan\phi,\tilde{F}_{\om}(z)\psi\ran|&\lesssim\sum_{|i-j|_{\infty}\leq R-1}\sqrt{|\om_{i}|}\sqrt{|\om_{j}|}\iint_{\R^{d}\times\R^{d}}\frac{|\phi(x)|\sqrt{|u_{i}(x)|}|\sqrt{|u_{j}(y)|}|\psi(y)|}{|x-y|^{d-2}}dxdy\\
&\quad\quad+\sum_{|i-j|_{\infty}\geq R}\sqrt{|\om_{i}|}\sqrt{|\om_{j}|}\iint_{\R^{d}\times\R^{d}}\frac{|\phi(x)|\sqrt{|u_{i}(x)|}|\sqrt{|u_{j}(y)|}|\psi(y)|}{|x-y|^{(d-2)/2}}dxdy\\
&=\tilde{I}_{\om}+\tilde{II}_{\om},
\end{split}
\end{equation*}
where we used the fact $|\tilde{V}_{\om}|^{1/2}\leq\sum_{i\in\Z^{d}}\sqrt{|\om_{i}|}\sqrt{|u_{i}|}$ and Lemma \ref{estimate-of-kernel} together with \eqref{disjoint-supp}.

Using $\rm(H4)$ and $\rm(H5)$, as in the proof of \eqref{estimate-1}, we deduce from H\"{o}lder's inequality that
\begin{equation}\label{key-estimate}
\begin{split}
\tilde{I}_{\om}&\lesssim\sum_{|i-j|_{\infty}\leq R-1}\|\chi_{C_{R}(i)}\phi\|\|\chi_{C_{R}(j)}\psi\|\\
&\leq\sum_{i\in\Z^{d}}\bigg(\|\chi_{C_{R}(i)}\phi\|\sum_{j:|i-j|_{\infty}\leq R-1}\|\chi_{C_{R}(j)}\psi\|\bigg)\\
&\leq\bigg(\sum_{i\in\Z^{d}}\|\chi_{C_{R}(i)}\phi\|^{2}\bigg)^{1/2}\bigg(\sum_{i\in\Z^{d}}\bigg(\sum_{j:|i-j|_{\infty}\leq R-1}\|\chi_{C_{R}(j)}\psi\|\bigg)^{2}\bigg)^{1/2}
\end{split}
\end{equation}
\begin{lem}\label{lem-aux-1}
\begin{itemize}
\item[\rm(i)] For any $r>0$, $\sum_{i\in\Z^{d}}\|\chi_{C_{r}(i)}\phi\|^{2}\lesssim_{r}\|\phi\|^{2}$.
\item[\rm(ii)] For odd $r\geq3$, $\sum_{j:|i-j|_{\infty}\leq r-1}\|\chi_{C_{r}(j)}\psi\|^{2}\lesssim_{r}\|\chi_{C_{3r-2}(i)}\psi\|^{2}$
\end{itemize}
\end{lem}
\begin{proof}
$\rm(i)$ follows from the fact that the function $\sum_{i\in\Z^{d}}\chi_{C_{r}(i)}$ is bounded with the bound depending on $r$. Similarly, $\rm(ii)$ follows from $\sum_{j:|i-j|_{\infty}\leq r-1}\chi_{C_{r}(j)}\lesssim_{r}\chi_{C_{3r-2}(i)}$.
\end{proof}
\begin{lem}\label{lem-aux-2}
$(\sum_{n=1}^{N}a_{n})^{2}\leq N\sum_{n=1}^{N}a_{n}^{2}$.
\end{lem}

Applying Lemma \ref{lem-aux-1} and Lemma \ref{lem-aux-2} to \eqref{key-estimate} gives
\begin{equation*}
\tilde{I}_{\om}\lesssim\|\phi\|\bigg(\sum_{i\in\Z^{d}}\sum_{j:|i-j|_{\infty}\leq R-1}\|\chi_{C_{R}(j)}\psi\|^{2}\bigg)^{1/2}\lesssim\|\phi\|\bigg(\sum_{i\in\Z^{d}}\|\chi_{C_{3R-2}(i)}\psi\|^{2}\bigg)^{1/2}\lesssim\|\phi\|\|\psi\|.
\end{equation*}

Using $\rm(H4)$, $\rm(H5)$ and \eqref{disjoint-supp} and estimating the expectation of the $L^{2}(\R^{d}\times\R^{d})$-norm of the integral kernel as in the proof of \eqref{estimate-2}, we deduce
\begin{equation*}
\E\Big\{\sup_{z\in\ol{\SS}}\|II_{\om,z}\|\Big\}\lesssim\|f\|_{\frac{2d}{d+1}},
\end{equation*}
where $f=\sum_{i\in\Z^{d}}p_{i}\chi_{C_{R}(i)}$. To compute the norm, we note
\begin{equation}\label{equality-10}
\sum_{i\in\Z^{d}}p_{i}\chi_{C_{R}(i)}=\sum_{i\in\Z^{d}}\bigg(\sum_{j\in C_{R}(i)}p_{j}\bigg)\chi_{C_{i}}\quad\text{on}\,\,\R^{d}\bs\Z^{d}.
\end{equation}
and
\begin{lem}\label{lem-decaying}
Let $\al>0$. Then, $p_{j}\sim_{R,\al}(1+|i|)^{-\al}$ for all $j\in C_{R}(i)$.
\end{lem}
\begin{proof}
Note
\begin{equation*}
p_{j}\sim\frac{1}{(1+|j|)^{\al}}=\frac{1}{(1+|i|)^{\al}}\bigg(\frac{1+|i|}{1+|j|}\bigg)^{\al}.
\end{equation*}
The lemma then follows since for $j\in C_{R}(i)$
\begin{equation*}
\frac{1+|i|}{1+|j|}\leq\frac{1+|j|+|i-j|}{1+|j|}\lesssim1+R\quad\text{and}\quad\frac{1+|j|}{1+|i|}\lesssim1+R.
\end{equation*}
\end{proof}

By \eqref{equality-10} and Lemma \ref{lem-decaying}, $\al>\frac{d+1}{2}$ ensures
\begin{equation*}
\|f\|_{\frac{2d}{d+1}}\lesssim\bigg\|\sum_{i\in\Z^{d}}\frac{\chi_{C_{i}}}{(1+|i|)^{\al}}\bigg\|_{\frac{2d}{d+1}}=\bigg(\sum_{i\in\Z^{d}}\frac{1}{(1+|i|)^{\frac{2d\al }{d+1}}}\bigg)^{\frac{d+1}{2d}}<\infty.
\end{equation*}
This proves \eqref{claim-10} and finishes the proof of Theorem \ref{thm-main-4}$\rm(i)$.

The second statement in Theorem \ref{thm-main-4} is a simple consequence of the first statement and Proposition \ref{prop-ac-absence}.


\appendix

\section{Absence of a.c. spectrum below zero}

We recall a result of Boutet de Monvel, Stollmann and Stolz (see \cite[Theorem 3.1]{BSS05}).

\begin{prop}\label{prop-ac-absence}
Consider the Schr\"{o}dinger operator with random potentials
\begin{equation*}
H_{\om}=H_{0}+\sum_{i\in\Z^{d}}\om_{i}u_{i}
\end{equation*}
satisfying the following conditions:
\begin{itemize}
\item[\rm(i)] $\{u_{i}\}_{i\in\Z^{d}}$ are real-valued functions and there's an open set $B$ containing $0$ such that $\supp(u_{i})\subset i+B$ for all $i\in\Z^{d}$;
\item[\rm(ii)] $\sup_{i\in\Z^{d}}\|u_{i}\|_{p}<\infty$ for $p\geq2$ if $d\leq 3$ and $p>\frac{d}{2}$ if $d>3$;
\item[\rm(iii)] $\{\om_{i}\}_{i\in\Z^{d}}$ are independent random variables on some probability space $(\Om,\BB,\P)$ such that there are $m<0<M$ such that $\P\{\om\in\Om|\om_{i}\in[m,M]\}=1$ for all $i\in\Z^{d}$ and for all $\ep>0$,
\begin{equation*}
\P\{\om\in\Om||\om_{i}|\geq\ep\}\sim(1+|i|)^{-\al},\quad\al>d-1.
\end{equation*}
\end{itemize}
Then, for a.e. $\om\in\Om$, $\si_{ac}(H_{\om})\cap(-\infty,0)=\emptyset$.
\end{prop}


\end{document}